\documentclass[12pt,reqno]{article}

\usepackage[usenames]{color}
\usepackage{amssymb}
\usepackage{graphicx}
\usepackage{amscd}

\usepackage[colorlinks=true,
linkcolor=webgreen,
filecolor=webbrown,
citecolor=webgreen]{hyperref}

\definecolor{webgreen}{rgb}{0,.5,0}
\definecolor{webbrown}{rgb}{.6,0,0}

\usepackage{color}
\usepackage{fullpage}
\usepackage{float}

\usepackage[american]{babel}
\usepackage{amsmath}
\usepackage{amsthm}
\usepackage{amsfonts}
\usepackage{skak}
\usepackage{numprint}

\newcommand{\seqnum}[1]{\href{https://oeis.org/#1}{\rm \underline{#1}}}

\begin{document}

\begin{center}
  \includegraphics[width=1in]{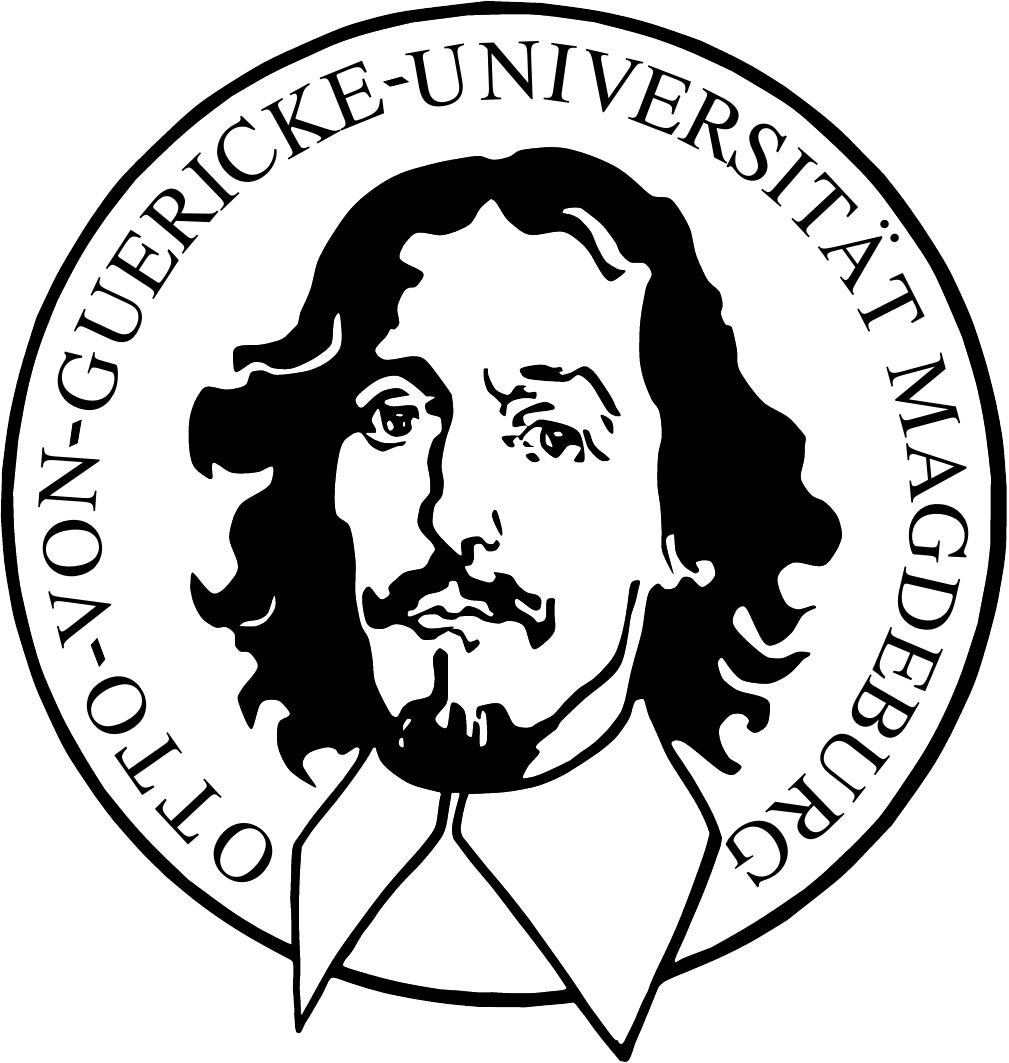} \hfill
  \includegraphics[width=1in]{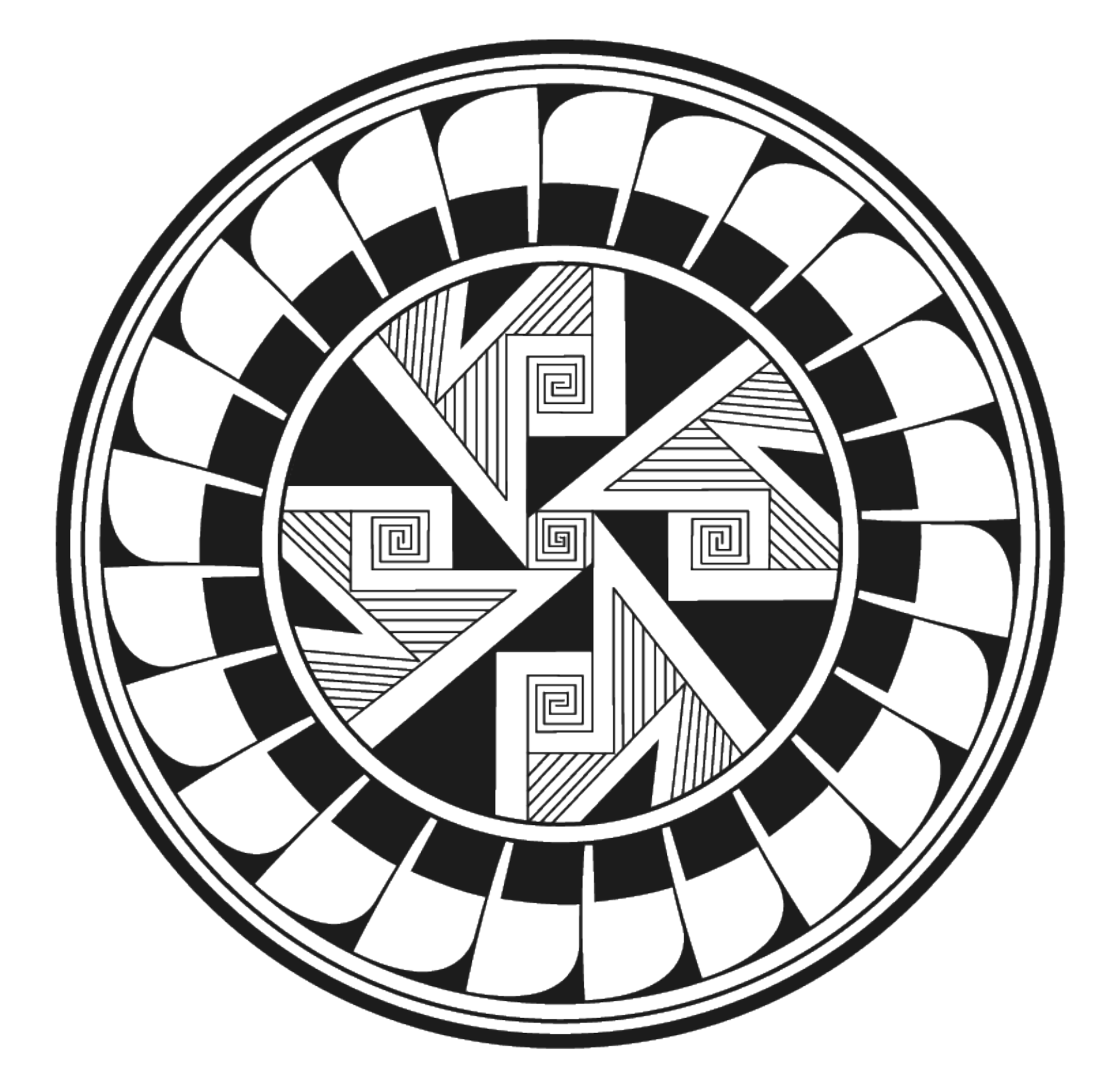}
\end{center}

\theoremstyle{plain}
\newtheorem{theorem}{Theorem}
\newtheorem{corollary}[theorem]{Corollary}
\newtheorem{lemma}[theorem]{Lemma}
\newtheorem{proposition}[theorem]{Proposition}

\theoremstyle{definition}
\newtheorem{definition}[theorem]{Definition}
\newtheorem{example}[theorem]{Example}
\newtheorem{conjecture}[theorem]{Conjecture}
\newtheorem{Observation}[theorem]{Observation}

\theoremstyle{remark}
\newtheorem{remark}[theorem]{Remark}

\begin{center}
\vskip 1cm{\LARGE\bf Domination Polynomial of the Rook Graph
}
\vskip 1cm
Stephan Mertens \\
\selectlanguage{ngerman} {Institut f\"ur Physik \\
Otto-von-Guericke Universit\"at Magdeburg \\
Postfach 4120, 39016 Magdeburg, Germany} \\
and \\
Santa Fe Institute\\
1399 Hyde Park Rd,
Santa Fe, NM 87501,
USA\\
\href{mailto:mertens@ovgu.de}{\texttt{mertens@ovgu.de}}
\end{center}

\vskip .2 in

\begin{abstract}
  A placement of chess pieces on a chessboard is called \emph{dominating} if each
free square of the chessboard is under attack by at least one
piece. In this contribution we compute the number of dominating
arrangements of $k$ rooks on an $n\times m$ chessboard. To this end we derive an
expression for the corresponding generating function, the
domination polynomial of the $n\times m$ rook graph.
\end{abstract}

\section{Introduction}

A placement of chess pieces on a chessboard is called
\emph{dominating} if each free square of the chessboard is under
attack by at least one piece. Chess domination problems have been
studied at least since 1862, when Jaenisch \cite{jaenisch:62} posed the
problem to find the minimum number of queens needed to dominate the
$8\times 8$ board. This number is known as the domination number
$\gamma_{\symqueen}$. The minimum number of knights needed to dominate
the $8\times 8$ board is called $\gamma_{\symknight}$. We can easily show
that $\gamma_{\symqueen} \leq 5$ and that $\gamma_{\symknight} \leq 12$
by presenting dominating placements of 5 queens and 12 knights
(Fig.~\ref{fig:8x8-examples}). Proving that both inequalities are
actually equations is more challenging \cite{watkins:12}.

\begin{figure}
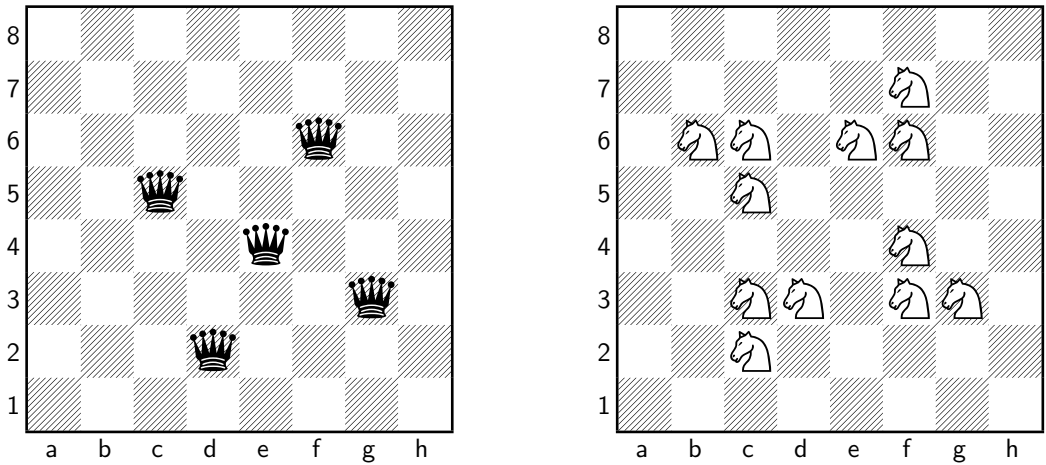

\begin{tabular}{p{0.45\textwidth}p{0.45\textwidth}}  
  \begin{center}
    \fenboard{8/8/5q2/2q5/4q3/6q1/3q4/8 - - - 0 0}
    $$\showboard$$
  \end{center} &
    \begin{center}
    \fenboard{8/5N2/1NN1NN2/2N5/5N2/2NN1NN1/2N5/8 - - - 0 0}
    $$\showboard$$ 
  \end{center}\\[-6ex]
\end{tabular}
\caption{5 queens or 12 knights can dominate the $8\times 8$ board.}
\label{fig:8x8-examples}
\end{figure}

The domination number for the $n\times n$ chessboard defines the
domination sequence. For queens, this sequence is
$\gamma_{\symqueen}(n) = 1, 1, 1, 2, 3, 3, 4, 5, 5, \ldots\,,$ which
is sequence \seqnum{A075458} in the On-Line Encyclopedia of Integer
Sequences (OEIS) \cite{oeis}.  For knights, the sequence reads
$\gamma_{\symknight}(n) = 1, 4, 4, 4, 5, 8, 10, 12, 14, \ldots$, which
is sequence \seqnum{A006075} in the OEIS. Both sequences are hard to
compute even for moderate values of $n$
\cite{fischer:03,kearse:gibbons:01,oestergard:weakley:01,rubin:03}.
The sequence $\gamma_{\symqueen}(n)$ is only known for $n \leq 25$,
and $\gamma_{\symknight}(n)$ for $n \leq 21$.

Domination on chessboards is a rich and active topic of research
\cite{cockayne:90,hedetniemi:hedetniemi:21}. The field got an additional boost
when it was extended to domination in graphs \cite{haynes:hedetniemi:slater:98}.
A subset $S\subseteq V$ of vertices in a graph $G=(V,E)$ is called a
dominating set if every vertex $v\in V$ is either an element of $S$ or
is adjacent to an element of $S$.

Chess domination problems can be recast in graph theory terms by
defining an appropriate graph. Take, as an example, queens on an
$n\times m$ board. In the queen graph $Q_{n,m}$ each vertex represents
a square on the chessboard. Two vertices $v$ and $u$ share an edge if
and only if a queen can legally move from $v$ to $u$.  The graphs
$N_{n,m}$ for knights and $R_{n,m}$ for rooks are defined
correspondingly.

A dominating placement of $k$ queens corresponds to dominating set of
cardinality $k$ in the graph $G=Q_{n,m}$. The problem of computing
$\gamma_{\symqueen}(n)$ corresponds to finding the minimum cardinality
of dominating sets in $Q_{n,n}$.

Let $d_G(k)$ denote the number of dominating sets in $G$ of
cardinality $k$. The domination polynomial $D_G(x)$ is defined as the
generating function of $d_G(k)$,
\begin{equation}
  \label{eq:def-D_G}
  D_G(x) = \sum_{k=\gamma_G}^{|V|}  d_G(k) x^k\,. 
\end{equation}
Like other graph polynomials, the domination polynomial encodes
many interesting properties of a graph \cite{akbari:alikhani:peng:10,alikhani:peng:14}. 

In this contribution we will compute the domination polynomial of the
rook graph $R_{n,m}$, which is the cartesian product of the complete
graphs $K_n$ and $K_m$.  Rook domination is considerably easier to
analyze than the domination of queens and knights. For example, the
domination sequence is given by the simple formula
\begin{equation}
  \label{eq:gamma-rook}
  \gamma_{R_{n,m}} = \min(n, m)\,.
\end{equation}
This follows from the fact that for domination, each row \emph{or}
each column must contain a rook.

Despite the simplicity of rook domination, very little is known about
the domination polynomial. Notable exceptions are its unimodality 
\cite{burcroff:obrien:23} and its lowest degree coefficient
\cite[problem 34b]{yaglom:yaglom:64}:
\begin{equation}
  \label{eq:R-gamma}
  d_{R_{n,m}}(\gamma_{R_{n,m}}) = \begin{cases}
    \max(n,m)^{\min(n,m)}, & \text{if $n \neq m$;}\\
    2n^n-n!, & \text{if $n = m$.}
    \end{cases}
  \end{equation}
  \begin{proof}
The case $n \neq m$ is obvious. For the square case we can place the
$n$ rooks to cover every row  ($n^n$ possibilities) or every column
(another $n^n$ possibilities). Adding these two numbers double counts
the configurations that cover both all columns and all rows. Hence we
need to subtract the number of those configurations, which is $n!$.
\end{proof}
To compute all the other coefficients, we will deploy the machinery of
generating functions. But before doing this, we will derive
a recursive equation that links $d_{R_{n,m}}(k)$ to enumerations in smaller
boards. 

\section{Recursion}

A dominating arrangement of rooks does not necessarily have a rook in
every column \emph{and} every row of the board. Think of $n$ rooks in the first
row, leaving all other $n-1$ rows empty. Arrangements of $k$ rooks that contain at
least one rook in every column \emph{and} every row are a subset of all
dominating configurations, and their number $E_{n,m}(k)$ is less than $d_{R_{n,m}}(k)$. We need
$E_{n,m}(k)$ to compute $d_{R_{n,m}}(k)$:
\begin{theorem} \label{thm:R-E}
  Let $E_{n,m}(k)$ denote the number of placements of $k$
  indistinguishable rooks on an $n\times m$ chessboard such that each
  row and each column contain at least one rook. Then
  \begin{equation}
    \label{eq:R-E}
    d_{R_{n,m}}(k) = {nm \choose k} - \sum_{r=1}^n \sum_{c=1}^m {n \choose
      r} {m \choose c} E_{n-r,m-c}(k)\,.
  \end{equation}
\end{theorem}
\begin{proof}
  The first term in \eqref{eq:R-E} is the number of all possible rook
  arrangements. Hence, we need to prove that the second term equals
  the number of non-dominating arrangements.

  In a non-dominating placement, at least one square is not attacked
  by any rook. This means, that the row and column of that square is
  void of rooks. So we need to have one or more empty rows and one
  or more empty columns. The second term in \eqref{eq:R-E} sums over all
  combinations of $r=1,\ldots,n$ empty rows and $c=1,\ldots,m$ empty
  columns. In order to avoid overcounting, each of the remaining $n-r$
  rows and $m-c$ columns must contain at least one rook. The number of
  the corresponding arrangements is given by $E_{n-r,m-c}(k)$.
\end{proof}
Theorem \ref{thm:R-E} allows us to compute $d_{R_{n,m}}(k)$ only if we know
how to compute $E_{n,m}(k)$, which seems to be as difficult as the
original task. The square case $E_{n,n}(k)$ can be found in the OEIS
as \seqnum{A055599}, but we need $E_{n,m}(k)$ for general $n$ and $m$.
Luckily, we can compute $E_{n,m}(k)$ by recursion:
\begin{theorem} \label{thm:E-recursion}
  With base case $E_{0,m}(k) = E_{n,0}(k) = 0$, the numbers
  $E_{n,m}(k)$ can be computed by recursion over $n$ and $m$:
  \begin{equation}
    \label{eq:E-recursion}
    E_{n,m}(k) = {nm \choose k} - \sum_{r=0}^n \sum_{c=0}^m {n \choose
      r} {m \choose c} E_{n-r,m-c}(k) (1-\delta_{0,r}\delta_{0,c})\,,
  \end{equation}
  where $\delta_{i,,j}$ is the Kronecker delta.
\end{theorem}
\begin{proof}
  The proof is almost identical to the proof of Theorem~\ref{thm:R-E},
  except that here the sums over $r$ and $c$ start at $0$. This is
  because even with all rows being covered ($r=0$), a configuration
  does not count if a single column is not covered ($c > 0$). And vice
  versa. The only case that needs to be excluded is $c=r=0$. This
  is the reason for the factor $(1-\delta_{0,r}\delta_{0,c})$.
\end{proof}

Theorems~\ref{thm:R-E} and \ref{thm:E-recursion} are
sufficient to compute $d_{R_{n,m}}(k)$ numerically. A literal
implementation of \eqref{eq:R-E} and \eqref{eq:E-recursion} in a
simple Python script computes $d_{R_{10,10}}(k)$ in a few seconds.
As a sanity check for an implementation one can compare
the numerical results to \eqref{eq:R-gamma} and to
the following ``high density'' formula:
\begin{corollary} \label{cor:high-density}
  For $k > nm-n-m-\min(n,m)+2$,
\begin{equation}
  \label{eq:R-high-density}
  d_{R_{n,m}}(k) = {nm \choose k} - nm {(n-1)(m-1) \choose k}\,. 
\end{equation}
\end{corollary}
\begin{proof}
  An unattacked square implies that its row and its column are void of
  rooks. One empty row and one empty column contain $m+n-1$ squares.
  If $k$ is larger than $nm-(n+m-1)=(n-1)(m-1)$, we have too many rooks on the
  board to clear a column and a row and all ${nm \choose k}$
  placements are dominating. The second binomial in
  \eqref{eq:R-high-density} is zero in this case, as it should be.

  If we want two unattacked squares we need to clear one row, one
  column and another row or column (whichever is shorter). This means
  $n+m+\min(n,m)-2$ empty squares.  
  For $k > nm-(n+m+\min(n,m)-2) $ we have again too many rooks on the
  board to achieve this. Hence we are left with a single unattacked square
  $(x,y)$, which can be anywhere on the board (factor $nm$). The $k$ rooks
  can be placed arbitrarily on the $nm-n-m+1$ squares
  other than row $x$ and column $y$, which explains the second
  binomial in \eqref{eq:R-high-density}.
\end{proof}

\section{The domination polynomial}
\label{sec:domination_polynomial}

Theorem~\ref{thm:R-E} tells us that we can compute the generating
function for $d_{R_{n,m}}(k)$ once we know the generating function for
$E_{n,m}(k)$. So let us have a closer look on the latter.

In the rook graph $R_{n,m}$, the vertices represent the squares on the
board.  There is another graph $K_{n,m}$, in which the \emph{edges}
represent the squares. Think of square $(x,y)$ as connecting row $x$
with column $y$. Hence, the vertices in $K_{n,m}$ are the rows and the
columns, and because each row is connected to each column by the square
in their intersection, $K_{n,m}$ is the complete bipartite graph.

The set of squares with rooks correspond to a subset of edges of
$K_{n,m}$, and each row and each column contain a rook if and only if
the corresponding edges are an \emph{edge cover}, i.e. a set $F$ of edges
such that each vertex of $K_{n,m}$ is adjacent to at least one $f\in
F$. Hence, $E_{n,m}(k)$ denotes the number edge coverings of
cardinality $k$ of the complete bipartite graph $K_{n,m}$. 
The corresponding generating function, the  \emph{edge cover polynomial} 
of $K_{n,m}$, is given by \cite[Corollary 5]{akbari:oboudi:13}
\begin{equation}
  \label{eq:bipartite-edge-coverings}
  \sum_{k=0}^{nm} E_{n,m}(k) x^k = \sum_{k=0}^m(-1)^{m-k}{m \choose k}((1+x)^k-1)^n\,.
\end{equation}

\begin{theorem} \label{thm:domination-polynomial}
  The domination polynomial of the $n\times m$ rook graph can be
  written as
  \begin{equation}
    \label{eq:domination-polynomial}
   D_{R_{n,m}}(x) = \big((1+x)^n-1\big)^m  - (-1)^m \sum_{k=0}^{m-1}
     (-1)^{k}{m \choose k}\big((1+x)^k-1\big)^{n}\,.
  \end{equation}
\end{theorem}
\begin{proof}
 Multiplication of \eqref{eq:R-E} by $x^k$ and summation over
 $k=0,\ldots,nm$ yields
 \begin{equation}
   \label{eq:proof-1}
   D_{R_{n,m}}(x) = \sum_{k=0}^{nm}{nm \choose k} x^k - \sum_{r=1}^n \sum_{c=1}^m {n \choose
      r} {m \choose c} \sum_{k=0}^{nm} E_{n-r,m-c}(k) x^k\,.
  \end{equation}
  The first term is $(1+x)^{nm}$. In the second term, the sum over $k$
  is the edge covering polynomial
  \eqref{eq:bipartite-edge-coverings}. Inserting these terms and
  changing the summation indices $r\mapsto n-r$ and $c\mapsto m-c$
  provide us with
  \begin{equation}
     \label{eq:proof-2}
     D_{R_{n,m}}(x) = (1+x)^{nm} - \sum_{r=0}^{n-1} \sum_{c=0}^{m-1} {n \choose
        r} {m \choose c}  \sum_{k=0}^{c}(-1)^{c-k}{c \choose k}\big((1+x)^k-1\big)^{r}\,.
   \end{equation}
   Using the identity
   \begin{equation}
     \label{eq:proof-aux}
     \sum_{r=0}^{n-1} {n \choose r} A^r = (1+A)^n - A^n
   \end{equation}
   with $A=(1+x)^k-1$, we can compute the sum over $r$ to obtain
   \begin{equation}
     \label{eq:proof-3}
     D_{R_{n,m}}(x) = \big((1+x)^n-1\big)^m  + \sum_{c=0}^{m-1}  {m
       \choose c}\sum_{k=0}^{c}(-1)^{c-k}{c \choose k}\big((1+x)^k-1\big)^{n}\,.
   \end{equation}
   In order to compute the sum over $c$, we change the order of
   summation,
   \begin{equation}
     \label{eq:marcs-trick}
     \sum_{c=0}^{m-1} \sum_{k=0}^c \cdots = \sum_{k=0}^{m-1} \sum_{c=k}^{m-1}\cdots\,,
   \end{equation}
   to get
    \begin{equation}
     \label{eq:proof-4}
     D_{R_{n,m}}(x) = \big((1+x)^n-1\big)^m  + \sum_{k=0}^{m-1}
     \big((1+x)^k-1\big)^{n} \sum_{c=k}^{m-1}(-1)^{c-k}{m
       \choose c}{c \choose k}\,.
   \end{equation}
   If the sum over $c$ would run from $k$ to $m$, it would evaluate to
   $0$, see \cite[Eq.~(5.24)]{graham:knuth:patashnik:94}. Hence
   \begin{equation}
     \label{eq:proof-aux-2}
     \sum_{c=k}^{m-1}(-1)^{c-k}{m
       \choose c}{c \choose k} = -(-1)^{m-k}{m \choose k}\,,
   \end{equation}
   which yields \eqref{eq:domination-polynomial}.
\end{proof}
Of course \eqref{eq:proof-2}, \eqref{eq:proof-3} and \eqref{eq:proof-4}
are also valid representations of the domination polynomial. It is a matter
of taste to choose \eqref{eq:domination-polynomial} as ``the''
  domination polynomial. Our choice was guided by the observation
  that the ``single sum'' form of \eqref{eq:domination-polynomial} is
  the most efficient for computations with Mathematica. With
  \eqref{eq:domination-polynomial}, the computation of
  $D_{R_{50,50}}(x)$ took about 2 minutes on a laptop. 

  A blemish of \eqref{eq:domination-polynomial} is that it does not
  display the symmetry $D_{R_{n,m}}(x) = D_{R_{m,n}}(x)$. But of
  course there is a variant that does:
  \begin{corollary} \label{cor:domination-polynomial}
  The domination polynomial of the $n\times m$ rook graph can also be
  written as
  \begin{equation}
    \label{eq:domination-polynomial-symmetric}
    D_{R_{n,m}}(x) = \big((1+x)^n-1\big)^m+\big((1+x)^m-1\big)^n- (-1)^{n+m}
    \sum_{\ell=0}^n \sum_{k=0}^m {n \choose \ell} {m \choose k} (-1)^{k+\ell}(1+x)^{k\ell}\!.
  \end{equation}
\end{corollary}
\begin{proof}
  Binomial expansion of $\big((1+x)^k-1\big)^{n}$ in
   \eqref{eq:domination-polynomial} provides us with
   \begin{equation}
     \label{eq:proof-6}
     D_{R_{n,m}}(x) = \big((1+x)^n-1\big)^m  - (-1)^{m+n} \sum_{\ell =
       0}^n{n \choose \ell} (-1)^\ell\sum_{k=0}^{m-1}
     (-1)^{k}{m \choose k}(1+x)^{k\ell}\,.
   \end{equation}
   The sum over $k$ can be computed according to \eqref{eq:proof-aux}:
   \begin{equation}
     \label{eq:proof-7}
      D_{R_{n,m}}(x) = \big((1+x)^n -1\big)^m + \big((1+x)^m-1\big)^n - (-1)^n
    \sum_{\ell=0}^n(-1)^\ell{n \choose \ell} \big((1+x)^\ell-1\big)^m\,.
  \end{equation}
  A binomial expansion of $\big((1+x)^\ell-1\big)^m$ yields \eqref{eq:domination-polynomial-symmetric}.
\end{proof}

\begin{table}
  \centering
  \begin{tabular}{rcl}
    $D_{R_{1,1}}(x)$ & $=$ & $x$ \\[1ex]
    $D_{R_{2,2}}(x)$ & $=$ & $\numprint{6}\,x^{2} + \numprint{4}\,x^{3} + x^{4}$\\[1ex]
    $D_{R_{3,3}}(x)$ & $=$ & $ \numprint{48}\,x^{3} + \numprint{117}
                             x^{4} + \numprint{126}\,x^{5} +
                             \numprint{84}\,x^{6} + \numprint{36}\,x^{7}
                             + \numprint{9}\,x^{8} + x^{9}$ \\[1ex]
    $D_{R_{4,4}}(x)$& $=$ & $\numprint{488}\,x^4 + \numprint{2640}\,x^5 +
                            \numprint{6712}\,x^6 + \numprint{10864}\,x^7 +
                            \numprint{12726}\,x^8 + \numprint{11424}\,x^9 +
                            \numprint{8008}\,x^{10}+$\\
                     & & $\numprint{4368}\,x^{11} + \numprint{1820}\,x^{12} + \numprint{560}\,x^{13} + \numprint{120}\,x^{14} + \numprint{16}\,x^{15} + x^{16}$ \\[1ex]
    $D_{R_{5,5}}(x)$& $=$ & $\numprint{6130}\,x^5 + \numprint{58300}\,x^6 + \numprint{269500}\,x^7 
                            + \numprint{808325}\,x^8 + \numprint{1778875}\,x^9 + \numprint{3075160}\,
                            x^{10} +$\\
                     & & $\numprint{4349400}\,x^{11} + \numprint{5154900}\,x^{12} +
                         \numprint{5186300}\,x^{13} + \numprint{4454400}\,x^{14} +
                         \numprint{3268360}\,x^{15} +$ \\
                     & & $\numprint{2042950}\,x^{16} + \numprint{1081575}\,x^{17} + \numprint{480700}\, 
                         x^{18} + \numprint{177100}\,x^{19} + \numprint{53130}\,x^{20} + $\\
                     & & $\numprint{12650}\,x^{21} + \numprint{2300}\,x^{22} + \numprint{300}\,x^{23} + 
                         \numprint{25}\,x^{24} + x^{25}$\\[1ex]
    $D_{R_{6,6}}(x)$ & $=$ & $\numprint{92592}\,x^{6} +
                             \numprint{1356480}\,x^{7} +
                             \numprint{9859140}\,x^{8} +
                             \numprint{47187180}\,x^{9} +
                             \numprint{167284836}\,x^{10} +$\\
                     & & $\numprint{469268496}\,x^{11} + \numprint{1086623400}\,x^{12} +
                         \numprint{2137381200}\,x^{13} + \numprint{3642777000}\,x^{14}
                         +$\\
                     & & $\numprint{5453014080}\,x^{15} + \numprint{7235196885}\,x^{16} +
                         \numprint{8558765100}\,x^{17} + \numprint{9057864300}\,x^{18}
                         +$\\
                     & & $\numprint{8591124600}\,x^{19} + \numprint{7305959610}\,x^{20} +
                         \numprint{5567447160}\,x^{21} + \numprint{3796214400}\,x^{22}
                         +$\\
                     & & $\numprint{2310778800}\,x^{23} + \numprint{1251676800}\,x^{24} +
                         \numprint{600805260}\,x^{25} + \numprint{254186856}\,x^{26} +$\\
                     & & $\numprint{94143280}\,x^{27} + \numprint{30260340}\,x^{28} +
                         \numprint{8347680}\,x^{29} + \numprint{1947792}\,x^{30} +
                         \numprint{376992}\,x^{31} +$\\
                     & & $\numprint{58905}\,x^{32} + \numprint{7140}\,x^{33} +
                         \numprint{630}\,x^{34} + \numprint{36}\,x^{35} + x^{36}$\\[1ex]
    $D_{R_{7,7}}(x)$  &$=$& $\numprint{1642046}\,x^{7} +
                            \numprint{34112526}\,x^{8} +
                            \numprint{355943644}\,x^{9} +
                            \numprint{2472314110}\,x^{10} +$\\
                     & &  $\numprint{12823222482}\,x^{11} +
                         \numprint{52933543012}\,x^{12} + \numprint{181178358774}\,
                         x^{13}+$\\
                     & & $\numprint{529116154896}\,x^{14} +
                         \numprint{1346298997554}\,x^{15}+
                         \numprint{3031523389181}\,x^{16} +$\\
                     & & $
                         \numprint{6112557579744}\,x^{17} +
                         \numprint{11134728203116}\,x^{18} +
                         \numprint{18446369091724}\,x^{19} +
                         $\\
                     & & $
                         \numprint{27928246211796}\,x^{20}+
                         \numprint{38781291222674}\,x^{21}+
                         \numprint{49515597595786}\,x^{22}+ 
                         $\\
                     & & $
                         \numprint{58230726508164}\,x^{23}+
                         \numprint{63144145569911}\,x^{24}+
                         \numprint{63175905655695}\,x^{25}+
                         $\\
                     & & $
                         \numprint{58330909718550}\,x^{26}+
                         \numprint{49695284721096}\,x^{27}+
                         \numprint{39048436087654}\,x^{28}+
                         $\\
                     & & $
                         \numprint{28277118318876}\,x^{29}+
                         \numprint{18851589456070}\,x^{30} +
                         \numprint{11554240013008}\,x^{31} +
                       $\\
                     & & $
                         \numprint{6499267511814}\,x^{32}+
                         \numprint{3348108643131}\,x^{33} +
                         \numprint{1575580671714}\,x^{34} +
                         $\\
                     & & $
                         \numprint{675248870772}\,x^{35} +
                         \numprint{262596783715}\,x^{36} +
                         \numprint{92263734836}\,x^{37} +$\\
                     & & $\numprint{29135916264}\,x^{38} + \numprint{8217822536}\,x^{39}
                         + \numprint{2054455634}\,x^{40} + \numprint{450978066}\,x^{41}
                         +$\\
                     & & $\numprint{85900584}\,x^{42} + \numprint{13983816}\,x^{43} +
                         \numprint{1906884}\,x^{44} + \numprint{211876}\,x^{45} +
                         \numprint{18424}\,x^{46} +$\\
                     & & $\numprint{1176}\,x^{47} + \numprint{49}\,x^{48} + x^{49}$
  \end{tabular}
  \caption{Domination polynomials of the $n\times n$ rook graph.}
  \label{tab:examples}
\end{table}

\begin{table}
  \centering
  \begin{tabular}{rcl}
    $D_{R_{8,8}}(x)$ & $=$ & $\numprint{33514112}\, x^{8} +
                             \numprint{933879296}\, x^{9} +
                             \numprint{13161955968}\, x^{10} +
                             \numprint{124392729216}\, x^{11} + $\\
    & & $\numprint{883565332160}\, x^{12} + \numprint{5020456535808}\,
        x^{13} + \numprint{23745692294080}\, x^{14} + $\\
    & & $\numprint{96124772710912}\, x^{15} + \numprint{339958097017896}\, x^{16} + \numprint{1067094188274240}\, x^{17} +$\\
    & & $ \numprint{3009775897325792}\, x^{18} +
        \numprint{7703325822650304}\, x^{19} +$\\
    & & $\numprint{18031600637765680}\, x^{20} +
        \numprint{38843543834346048}\, x^{21} +$\\
    & & $ \numprint{77392553377032096}\, x^{22} + \numprint{143185055260371264}\, x^{23} + $\\
    & & $\numprint{246761069109093336}\, x^{24} + \numprint{397106882820897536}\, x^{25} + $\\
    & & $\numprint{597898212185747424}\, x^{26} + \numprint{843500295460142656}\, x^{27} +$\\
    & & $ \numprint{1116294749822105392}\, x^{28} + \numprint{1387019957382904768}\, x^{29} + $\\
    & & $\numprint{1619086454915331808}\, x^{30} + \numprint{1776352520871483072}\, x^{31} + $\\
    & & $\numprint{1832208846791514422}\, x^{32} + \numprint{1776875996843390912}\, x^{33} + $\\
    & & $\numprint{1620187226242379648}\, x^{34} + \numprint{1388775090898717312}\, x^{35} + $\\
    & & $\numprint{1118753489141190336}\, x^{36} + \numprint{846631073977386432}\, x^{37} + $\\
    & & $\numprint{601555988478702432}\, x^{38} + \numprint{401038042815966528}\, x^{39} + $\\
    & & $\numprint{250648973984891272}\, x^{40} + \numprint{146721398729422272}\, x^{41} + $\\
    & & $\numprint{80347442945600992}\, x^{42} + \numprint{41107995982971456}\, x^{43} + $\\
    & & $\numprint{19619725660610544}\, x^{44} +
        \numprint{8719878112062656}\, x^{45} +$\\
    & & $\numprint{3601688789838944}\, x^{46} +
        \numprint{1379370175208256}\, x^{47} +$\\
    && $\numprint{488526937076444}\, x^{48} + \numprint{159518999862656}\, x^{49} + $\\
    & & $\numprint{47855699958816}\, x^{50} + \numprint{13136858812224}\, x^{51} + \numprint{3284214703056}\, x^{52} + $\\
    & & $\numprint{743595781824}\, x^{53} + \numprint{151473214816}\, x^{54} + \numprint{27540584512}\, x^{55} + $\\
    & & $\numprint{4426165368}\, x^{56} + \numprint{621216192}\, x^{57} + \numprint{74974368}\, x^{58} + \numprint{7624512}\, x^{59} + $\\
    & & $ \numprint{635376}\, x^{60} + \numprint{41664}\, x^{61} +\numprint{2016}\, x^{62} + \numprint{64}\, x^{63} + x^{64}$
  \end{tabular}
  \caption{Domination polynomial of the $8 \times 8$ rook graph.}
  \label{tab:8x8}
\end{table}

Tables~\ref{tab:examples} and \ref{tab:8x8} show the domination polynomials
  $D_{R_{n,n}}(x)$ for $n=1,\ldots,8$. The coefficients of
  $D_{R_{n,n}}$ have become sequence \seqnum{A368831} in 
  the OEIS. The \emph{total number} of dominating sets,
\begin{equation}
  \label{eq:total}
  D_{R_{n,m}}(1) = (2^n-1)^m + (2^m-1)^n
  -(-1)^{n+m}\sum_{\ell=0}^n\sum_{k=0}^m (-1)^{k+\ell}2^{k\ell}\,,
\end{equation}
is in the OEIS as \seqnum{A287274}.

\section{Conclusions}

The connection between the domination polynomial of the rook graph
$R_{n,m}$ and the edge cover polynomial of the complete bipartite
graph $K_{n,m}$ allowed us to compute the former.
Theorem~\ref{thm:domination-polynomial} is our main result. As far as
we know, the rook is the first chess piece for which the domination
polynomial has been computed.

Evaluating the domination polynomial with a computer algebra system
like Mathematica seems to be the fastest way to compute the numerical
values of $d_{R_{n,m}}(k)$. These values have applications in
cryptography \cite{beguinot:etal:24}, which was the initial motivation
for this work.

The domination polynomial can also be used to study structural
properties of the sequences $d_{R_{n,m}}(k)$, like unimodality (which
has been proven recently using general arguments
\cite{burcroff:obrien:23}), the maximum,
or the asymptotics for large board sizes. We
leave this for further studies.

\section{Acknowledgments}

I am grateful to Sylvain Guilley and Julien B\'{e}guinot for guiding
my attention to this problem and to Cris Moore for helpful
discussions.  In theory, this work would have been possible without
the On-Line Encyclopedia of Integer Sequences. In practice, it was
not.

\bibliographystyle{jis}


\bigskip
\hrule
\bigskip

\noindent 2020 {\it Mathematics Subject Classification}:
Primary
  05C69   
  
Secondary
  05A15,  
  05C30,  
  11B83   
 
\noindent \emph{Keywords:}
domination polynomial, dominating set, rook graph

\bigskip
\hrule
\bigskip

\noindent (Concerned with sequences \seqnum{A006075},
\seqnum{A055599}, \seqnum{A075458} , \seqnum{A287274}, \seqnum{A368831})

\end{document}